\documentclass[letter, 10pt, conference]{ieeeconf}
\IEEEoverridecommandlockouts      
\author{Shuang Gao and Peter E. Caines% <-this % stops a space
\thanks{*This work is supported in part by NSERC (Canada), and the U.S. ARL and ARO grant W911NF1910110.}% <-this % stops a space
\thanks{Shuang Gao and Peter E. Caines are with the Department of Electrical and Computer Engineering, McGill University,
  Montreal, QC, Canada. \hspace{1cm}
Email:        {\tt\small    $\{$sgao,peterc$\}$@cim.mcgill.ca}. }%
}

\usepackage[noadjust,sort,compress]{cite}
\usepackage{amsmath}
\usepackage[caption=false,font=footnotesize]{subfig}
\usepackage{url}
\usepackage{color}
\usepackage{graphicx}
\usepackage{epstopdf}

\usepackage{comment}
\includecomment{comment}
\specialcomment{notes}{\begingroup\sffamily\color{blue}}{\endgroup}
\newcommand{\BA}{\mathbb{A}}
\newcommand{\BB}{\mathbb{B}}
\newcommand{\BR}{\mathds{R}}  % Space of Real Line
\newcommand{\BI}{\mathbb{I}}

\newcommand{\FA}{\mathbf{A}}   % Function A
\newcommand{\FB}{\mathbf{B}}   % Function B
\newcommand{\Fu}{\mathbf{u}}	   % Function u
\newcommand{\Fx}{\mathbf{x}}	   % Function x
\newcommand{\Fv}{\mathbf{v}}	   % Function v
\newcommand{\I}{\mathbb{I}}			% Identity operator I 
\newcommand{\FQ}{\mathbf{Q}}			% Identity operator I 
\newcommand{\Ff}{\mathbf{f}}	   % Function f

\newcommand{\FP}{\mathbf{P}}	   % BF s
\newcommand{\FL}{\mathbf{L}}
\newcommand{\FM}{\mathbf{M}}

\newcommand{\SA}{\mathbf{A^{[N]}}}
\newcommand{\SB}{\mathbf{B^{[N]}}}

\newcommand{\Sxt}{\mathbf{x^{[N]}_t}}

\newcommand{\DSxt}{\mathbf{\dot{x}^{[N]}_t}}

\newcommand{\Sut}{\mathbf{u^{[N]}_t}}
\newcommand{\Chi}{\mathds{1}}
\newcommand{\GST}{\mathbf{\tilde{G}^{sp}}}

\newcommand{\GSOT}{\mathbf{
\tilde{G}_1^{sp}}}

\newcommand{\TJ}{\breve{J}}  % Auxillary Cost function

\usepackage{amssymb}

\usepackage[T1]{fontenc}
\usepackage{dsfont}

\usepackage[standard,amsmath,thmmarks]{ntheorem}
\newtheorem{assumption}{Assumption}

\newcommand*\TRANS{{\mathpalette\doTRANS\empty}}
\makeatletter
\newcommand*\doTRANS[2]{\raisebox{\depth}{$\m@th#1\intercal$}}
\makeatother
\title{Optimal and Approximate Solutions to Linear Quadratic Regulation of a Class of Graphon Dynamical Systems}
\begin{document}

\maketitle

\begin{abstract}

%Controlling extremely large networks of coupled systems is motivated by applications such as smart grid, the Internet of Things and epidemic networks.

In this paper we study the linear quadratic regulation (LQR) problem for dynamical systems coupled over large-scale networks and obtain  locally computable low-complexity solutions. The underlying large or even infinite networks are represented by graphons and the couplings appear in both the dynamics and the quadratic cost. 
The optimal solution is obtained first for graphon dynamical systems for the special case where the graphons are exactly characterized by finite spectral summands. 
The complexity of generating these control solutions involves solving $d+1$ scalar Riccati equations where $d$ is the number of non-zero eigenvalues in the spectral representation.
 Based on this, we provide a suboptimal low-complexity solution for problems with general graphon couplings via spectral approximations and analyze the performance under the approximate control.
Finally, a numerical example is given to illustrate the explicit solution and demonstrate the simplicity of the solution. 

\end{abstract}
\section{Introduction}
Complex networks of dynamical systems arise in many applications such as the Internet of Things, 5G communications, grid networks, social interactions, epidemic networks, and biological neuronal networks. There is an obvious need to analyze and control such networks~\cite{liu2011controllability,chu2017complex,nowzari2016analysis}.  
The study of controlling complex networks typically involves the control analysis such as controllability~\cite{liu2011controllability}, control energy~\cite{pasqualetti2014controllability}, input node selection~\cite{chen2017pinning},
and the low-complexity control synthesis problems with simplified objective (e.g. consensus \cite{olfati2007consensus} or synchronization \cite{arenas2008synchronization}), simplified control (e.g. pinning control \cite{chen2017pinning}, ensemble control \cite{li2011ensemble}), low-rank (e.g. mean field) coupling \cite{yong2013linear,arabneydi2016team,zecevic2005global}, or patterned coupling \cite{hamilton2012patterned}.

Graphon theory is developed to model large networks and graph limits \cite{lovasz2012large}. It has been applied to study dynamical models such as the heat equation and the coupled oscillator model  \cite{chiba2019mean,medvedev2014nonlinear,Kuehn2018PowerNetsGraphon}.
Graphon-based control has recently been proposed and developed to study control problems of large-scale network-coupled dynamical systems and generate low-complexity approximate control solutions to the otherwise intractable problems \cite{ShuangPeterCDC17, ShuangPeterCDC18}, which follows the spirit of mean field games \cite{HCM07}. It employs the graphon model and infinite dimensional system theory \cite{bensoussan2007representation} to represent systems on networks of arbitrary sizes.
Graphon couplings can  also be considered as the generalization of mean-field couplings.

In applications involving dynamical systems coupled over a large-scale network, it is natural that not only the states, but also
controls and costs are  coupled via the underlying structure given by the network.
This paper provides explicit optimal and approximate solutions to the control of a class of graphon dynamical systems where the graphon couplings may appear in both the dynamics and the cost function. 
 Furthermore, the proposed solution can be implemented in a distributed manner. The complexity of generating the optimal control depends on the number of nonzero eigenvalues of the underlying graphon. 

\subsection*{Notation}
We use $\FA^\TRANS$ to denote the adjoint operator of $\FA$. $\BR$ and $\BR_+$ denote the set of all real numbers and that of all positive reals respectively. Bold face letters (e.g. $\FA$, $\FB$, $\Fu$, $\GSOT$) are used to represent graphons, functions, or graphon spaces. Blackboard bold letters (e.g. $\BA$, $\BB$) are used to denote linear operators which are not necessarily compact. Let $\I$ denote the identity operator.  We use $\langle\cdot ,  \cdot \rangle$ and $\|\cdot \|$ to represent respectively inner product and norm.
In this paper, unless stated otherwise, the term "graphon" refers to a symmetric measurable function  ${\FA_1}: [0,1]^2\rightarrow [-1,1]$ and  $\bf{\tilde{G}}_1^{sp}$ denotes the set of all graphons.
For any $c\in \BR_+$, let $\GST_c$ denote the set of all bounded symmetric measurable functions $\FA: [0,1]^2 \rightarrow [-c,c]$. 
Clearly any $\FA \in \GST_c$ can be interpreted as a linear operator from $L^2_{[0,1]}$ to $L^2_{[0,1]}$ (see e.g. \cite{ShuangPeterTAC18}). $\mathcal{L}(L^2_{[0,1]})$ shall denote the set of all bounded linear operators from $L^2_{[0,1]}$ to $L^2_{[0,1]}$.
 $\mathcal{PO}(\FA)$ will denote the set of all bounded linear operators which are polynomials of the graphon operator $\FA$. Note that $\I$ is an element of $\mathcal{PO}(\FA)$.

\section{System Model: Non-compact Operator Formulation}

\subsection{Linear Graphon Dynamical  Systems}

Let $\BA = (\alpha_0 \I + \FA)$ with $\FA \in \GSOT$. Then $\BA$ is a bounded linear operator from $L^2_{[0,1]}$ to $L^2_{[0,1]}$ with the operator action defined as
$$
[\BA \Fv](\cdot) = \alpha_0 \Fv(\cdot) + \int_0^1\FA(\cdot, \eta) \Fv(\eta) d\eta, \quad \Fv \in L^2_{[0,1]}.
$$
% Let $\BB$ be a bounded linear operator from $L^2_{[0,1]}$ to $L^2_{[0,1]}$.
 Following \cite{pazy1983semigroups}, $\BA$ is  the infinitesimal generator of the uniformly (hence strongly)  continuous semigroup 
 $S_\BA(t):=e^{\BA t}=\sum_{k=0}^{\infty} \frac{t^k\mathbf{\BA}^k}{k!}, ~ 0\leq t <\infty.$
Therefore, the initial value problem of the graphon  differential equation
\begin{equation}
	\mathbf{\dot{y}_t}={\BA\mathbf{y}_t},\quad \mathbf{y_0} \in L^2_{[0,1]} \label{equ:Noncompact-Operator-Differential-Equation}, \qquad  0\leq t <\infty,
\end{equation}
is well defined and has a  solution given by 
$
	\mathbf{y}_t=e^{{\BA}t}\mathbf{y}_0.
$

We formulate the graphon linear system $(\BA;\BB)$ as follows: 
\begin{equation} \label{equ:infinite-system-model}
	 \begin{aligned}
		&\mathbf{\dot{x}_t}= \BA \Fx_t + \BB \Fu_t,~~ t \in [0,T], % \quad \mathbf{x_0} \in L^2[0,1], 	
		\end{aligned}
\end{equation}
where $\BA = (\alpha_0 \BI +\FA)$ with  $\FA \in \GSOT$ and $\alpha_0 \in \BR$, $\BB \in \mathcal{L}(L^2_{[0,1]})$, ${\Fx_t} \in L^2_{[0,1]}$ is the system state at time $t$, and ${\Fu_t} \in L^2_{[0,1]}$ is the control input at time $t$. 
We limit our discussions to graphons $\FA \in \GSOT$ purely for simplicity.  The generalization to functions $\FA\in \GST_c$ is immediate.

 Let $C([0,T];L^2_{[0,1]})$ denote the set of continuous mappings from $[0,T]$ to $L^2_{[0,1]}$ and further let $L^2 ([0, T]; L^2_{[0,1]} )$ denote the Banach space of equivalence classes
of strongly measurable (in the B\"ochner sense \cite[p.103]{showalter2013monotone}) mappings $\Fx:[0,T] \rightarrow  L^2_{[0,1]} $ that are integrable with the norm $\|\Fx\|_{L^2 ([0, T]; L^2_{[0,1]})}= (\int_0^T \int_0^1\Fx_{\tau}(\alpha)^2d\alpha d\tau )^{\frac12}$.
A  solution $\Fx \in L^2([0,T];L^2_{[0,1]})$  is called a {\it mild solution} of  (\ref{equ:infinite-system-model}) if 
$	\Fx_t=e^{(t-a){\BA}}{\Fx_a}+\int_a^te^{(t-s)\BA}\BB\Fu_sds$
for all $a\leq t$ in $[0,T]$.

\begin{proposition}
The system $(\BA; \BB)$ in
\eqref{equ:infinite-system-model} has a unique  mild solution $\mathbf{x}\in C([0,T];L^2_{[0,1]})$ for any $\mathbf{x_0} \in L^2_{[0,1]}$ and any $\mathbf{u}\in L^2([0,T]; L^2_{[0,1]})$.
\end{proposition}
\begin{proof}
Since $\BA$ generates a strongly continuous semigroup and $\BB$ is a bounded linear operator on $L^2_{[0,1]}$, we obtain this result following \cite[p.385]{bensoussan2007representation}.

\end{proof}

\subsection{Relation to Finite Network Systems}

%\subsection{Network System Model}

Consider an interlinked network of linear (symmetric) dynamical %control
 subsystems $\{S_i^N; 1\leq i \leq N \}$. 
 The subsystem $S_i^N$ at the node $i$ in the undirected weighted graph $G_N$ has  interactions with $S_j^N, 1\leq j \leq N,$ specified as below:
 \vspace{-0.2cm}
 \begin{equation} \label{equ:network-system}
 	\begin{aligned}
 	 &\dot{x}^i_t= \alpha_0 x_t^i+ \frac{1}{N}\sum_{j=1}^{N} {a}_{ij} x^j_t+\beta_0 u_t^i +\frac1{N} \sum_{j=1}^N {b}{_{ij}} u^j_t, \quad \\
 	& t \in [0,T],\quad \alpha_0, \beta_0 \in \BR, \quad  x^i_t, u^i_t \in \BR,
 	\end{aligned}
 \end{equation}  
  where ${A}_{N}= [{a}_{ij}]$ and $ {B}_{N} =[{b}_{ij}] \in \BR^{N\times N}$ are the symmetric  adjacency matrices of $G_N$ and  of the input graph with bounded elements, say, $ |a_{ij}|, |b_{ij}| \leq 1$. {This bound can be generalized to a constant $c\in\BR_+$ if we work in $\GST_c$.}  
  For simplicity, the scalar state for each subsystem is considered here and this easily generalizes to vector state cases.  
   Let $x_t=[x^1_t, \dots, x^N_t]^\TRANS$ and $u_t=[u^1_t, \dots, u^N_t]^\TRANS$.

Consider a uniform partition $\{P_1, \ldots,P_N\}$ of $[0,1]$.  
Let the step function graphon $\SA$ that corresponds to $A_N$ be given by 
\begin{equation*}
	\SA(\vartheta,\varphi) = \sum_{i=1}^{N} \sum_{j=1}^{N} \Chi_{_{P_i}}(\vartheta)\Chi_{_{P_j}}(\varphi)a_{ij},  \quad (\vartheta,\varphi) \in [0,1]^2,
\end{equation*}
where $\Chi_{_{P_i}}(\cdot)$ represents the indicator function, that is, $\Chi_{_{P_i}}(\vartheta)=1$ if $\vartheta\in P_i$ and $\Chi_{_{P_i}}(\vartheta)=0$ if $\vartheta\notin P_i$. Similarly, define $\SB$ based on $B_N$.
Let the piece-wise constant function $\Sxt \in L^2_{[0,1]}$ corresponding to $x_t \in \BR^N$ be given by $\Sxt (\vartheta) =\sum_{i=1}^N {\Chi}_{_{P_i}}(\vartheta) x_t^i$, for all $\vartheta \in [0,1].$ 
Similarly define $\Sut \in L^2_{[0,1]}$ that corresponds to $u_t\in \BR^N$.

Then the corresponding graphon dynamical system for the network system in \eqref{equ:network-system} is given by 
\begin{equation}
	\begin{aligned} \label{equ:step-function-dynamical-system}
	&\DSxt= (\alpha_0 \BI+\SA) \Sxt+(\beta_0 \BI+\SB) \Sut, \quad t\in[0,T],\\
	&\alpha_0, \beta_0 \in \BR,   \quad \Sxt, \Sut \in L^2_{pwc}{_{[0,1]}},\quad  \SA, \SB \in \GSOT
	\end{aligned}
\end{equation}
 where $L^2_{pwc}{_{[0,1]}}$ represents the set of all piece-wise constant functions in $L^2_{[0,1]}$.

 The trajectories of the graphon dynamical system in \eqref{equ:step-function-dynamical-system} correspond one-to-one to the trajectories of the network system in \eqref{equ:network-system}. 
Moreover, the system in \eqref{equ:infinite-system-model} can represent the limit system for a sequence of systems represented in the form of \eqref{equ:step-function-dynamical-system} when the underlying step function graphon sequences convergence in the $L^2_{[0,1]^2}$ metric~\cite{ShuangPeterTAC18}.

\section{Optimal Control Problem} \label{sec:ocp}

\subsection{Control Objective}
Consider the instantaneous cost 
$
c_t(\Fu_t,\Fx_t) = \langle \Fx_t, \FQ \Fx_t \rangle+ \langle \Fu_t, \Fu_t\rangle,
$
and the terminal cost
$
	c_T(\Fx_T) = \langle \Fx_T, \FP_0 \Fx_T \rangle. %\quad \FP_0\geq 0,
$ % where $\langle \cdot, \cdot \rangle$ denotes the inner product in $L^2_{[0,1]}$. 
The control objective is to obtain the control law $\Fu \in L^2([0,T]; L^2_{[0,1]})$ that minimizes the quadratic cost
\begin{equation}\label{equ:original-cost-function}
	J(\Fu) =\int_0^T c_t(\Fx_t, \Fu_t) dt + c_T(\Fx_T),
\end{equation}
 subject to the system dynamics in \eqref{equ:infinite-system-model} over the finite time horizon $[0, T]$.

\subsection{Existence and Uniqueness of Optimal Solutions}
%Denote $\BA= \alpha_0 \I +\FA$.
Consider the following Riccati equation
\begin{equation}\label{equ: noncompact-Riccati-Equation}
	\dot{\FP}=\BA^\TRANS \FP+\FP \BA- \FP\BB\BB^\TRANS \FP + \FQ, \quad \FP(0)= \FP_0.
\end{equation}
Given the solution $\FP$ to the Riccati equation, the optimal control $\Fu^*:=\{\Fu^*_t, t\in[0,T]\}$ is given by 
\begin{equation} \label{equ: noncompact-feedback-control-law}
	\Fu^*_t=-\BB^\TRANS \FP(T-t) \Fx^*_t, \quad t\in [0, T]
\end{equation}
and moreover $\Fx^*:=\{\Fx^*_t, t\in[0,T]\}$ is the solution to the closed loop equation 
\begin{equation} \label{equ: Closed-Loop-System-Non-Compact}
	\begin{aligned}
	& \dot{\Fx}_t=\big(\BA- \BB \BB^\TRANS \FP(T-t)\big) \Fx_t, 
	& t\in [0,T], \Fx_0 \in L^2_{[0,1]}.
	\end{aligned}
\end{equation}
See \cite{bensoussan2007representation} for more details. 

\begin{assumption} \label{ass:graphon-team-2}
	The linear operators $\FQ$ and $\FP_0$ on $L^2_{[0,1]}$ are Hermitian and non-negative, i.e., $\FQ,\FP_0 \geq 0$.
\end{assumption}

\begin{proposition}[{{\cite[p.385]{bensoussan2007representation}}}] 
\label{prop:existence-unique-opt-sol}
	Under  \textbf{Assumption \ref{ass:graphon-team-2}}, there exists a unique solution to the Riccati equation  \eqref{equ: noncompact-Riccati-Equation} and furthermore there exists a unique optimal solution pair $(\Fu^*, \Fx^*)$ as given in \eqref{equ: noncompact-feedback-control-law} and \eqref{equ: Closed-Loop-System-Non-Compact}.
\end{proposition}

\begin{assumption} \label{ass:graphon-team-3}
	The graphon $\FA$ as an operator has a finite  number $d$ of eigenfunctions corresponding to the finite set of non-zero eigenvalues. That is, 
	\begin{equation} 
	\FA(x,y) = \sum_{\ell=1}^{d} \lambda_\ell \Ff_\ell(x) \Ff_\ell(y), \quad (x,y)\in[0,1]^2.
\end{equation}
\end{assumption}

As an operator any graphon is compact and hence its eigenvalues accumulate at zero \cite{lovasz2012large}. Thus the above assumption corresponds to an reasonable approximation. See Section \ref{sec:graphon-approx} and \cite{ShuangPeterCDC19W1} for detailed discussions on graphon approximations. 
%We pose the following assumptions. %on the structure of the input operator 
\begin{assumption} \label{ass:system-operator-B}
	$\BB $ is in $ \mathcal{PO}(\FA)$ and it is given by
$
	\BB =\text{\normalfont poly}_\FB(\FA) := \sum_{k=0}^{b_L} \beta_k\FA^k, ~ b_L\geq0.
$
\end{assumption}
%The following extra assumptions are posed.

\begin{assumption} \label{ass:graphon-team-4}
$\FQ$ and $\FP_0$ are in $\mathcal{PO}(\FA)$, represented by
$\FQ  = \text{\normalfont poly}_\FQ(\FA):= \sum_{k=0}^h q_k\FA^k,~ h\geq0$ and $\FP_0 = \text{\normalfont poly}_{\FP_0}(\FA):=  \sum_{k=0}^{r}z_k\FA^k, ~ r\geq0$.
\end{assumption}

For systems coupled over a graphon, it is reasonable or even desirable in some applications that controls or costs are also coupled via the underlying structure given by the graphon. Notice that \textbf{Assumptions \ref{ass:system-operator-B}-\ref{ass:graphon-team-4}} include the cases with decoupled costs and controls. 

For any $ s\in \BR$, $\text{\normalfont poly}_\FB(s):= \sum_{k=0}^{b_L} \beta_k s^k;$ similar definitions hold for $\text{\normalfont poly}_\FQ(s)$ and $\text{\normalfont poly}_{\FP_0}(s)$.

\section{Optimal Solutions via  Decoupling}
 %\subsubsection*{Notations}
The projections of $\Fx_t$ and $\Fu_t$ in the eigenspace spanned by the normalized eigenfunction $\Ff_\ell \in L^2_{[0,1]}$ are respectively given by
$
  	\bar{\Fx}_t^\ell = \langle \Fx_t , \Ff_\ell \rangle \Ff_\ell \in L^2_{[0,1]} ~ \text{ and }~ \bar{\Fu}_t^\ell = \langle \Fu_t , \Ff_\ell \rangle \Ff_\ell \in L^2_{[0,1]}.
  $
   We call $\bar{\Fx}_t^\ell$ and $\bar{\Fu}_t^\ell$ eigenstate and eigencontrol, respectively.
 The  values for the respective inner products are denoted by
 \begin{equation}\label{eq:scalar-inner-product}
 	\bar{x}_t^\ell=\langle \Fx_t, \Ff_\ell \rangle \in \BR ~~~\text{and}~~~ \bar{u}_t^\ell=\langle \Fu_t, \Ff_\ell \rangle \in \BR.
 \end{equation}
To orthogonally decouple the optimal control problem, an auxiliary state and an auxiliary control are introduced as follows: 
\begin{equation}\label{equ:aux-state-control}
	\breve{\Fx}_t= \Fx_t -\sum_{\ell=1}^{d} \bar{\Fx}_t^\ell \quad \text{and} \quad \breve{\Fu}_t= \Fu_t -\sum_{\ell=1}^{d} \bar{\Fu}_t^\ell.
\end{equation}
\subsection{Decoupled Dynamics}

\begin{lemma}\label{lem:decoupling-dynamics}
Under \textbf{Assumptions \ref{ass:graphon-team-3}} and \textbf{\ref{ass:system-operator-B}}, the original system dynamics in \eqref{equ:infinite-system-model} can be uniquely decoupled into
the auxiliary system dynamics given by 
\begin{equation}\label{equ:aux-system}
	\dot{\breve{\Fx}}_t= \alpha_0 \breve{\Fx}_t+ \beta_0 \breve{\Fu}_t
\end{equation}
and the eigensystem dynamics given by
	\begin{equation}\label{equ:eigen-system}
	\dot{\bar{\Fx}}_t^\ell = (\alpha_0 + \lambda_\ell) \bar{\Fx}^\ell_t + \text{\normalfont poly}_{\FB}(\lambda_\ell) \bar{\Fu}_t^\ell,\quad 1\leq \ell \leq d. 
\end{equation}
\end{lemma}
\begin{proof}
By projecting both sides of \eqref{equ:infinite-system-model} into the direction  $\Ff_\ell$, we obtain \eqref{equ:eigen-system}. Then by subtracting \eqref{equ:eigen-system} for all $\ell$, $1\leq \ell \leq d$ from \eqref{equ:infinite-system-model} according to the definitions of the auxiliary state and the auxiliary control in \eqref{equ:aux-state-control}, we obtain \eqref{equ:aux-system}.

\end{proof}
\subsection{Decoupled Costs}
\begin{lemma} \label{lem:seperation-in-quadratic-functions}
	%Assume $\FQ=\text{\normalfont poly}(\FA)$ with $\text{\normalfont poly}(\cdot)$ representing polynomial functions.
	Under \textbf{Assumption \ref{ass:graphon-team-3}} and the assumption $\FQ \in \mathcal{PO}(\FA)$,  the following decoupling holds
	\begin{equation}\label{equ:cost-seperation}
		\begin{aligned}
			\langle \Fx_t, \FQ \Fx_t \rangle = \langle  \breve{\Fx}_t, \FQ \breve{\Fx}_t \rangle+  \sum_{\ell=1}^d \langle  {\bar{\Fx}}^\ell_t, \FQ \bar{\Fx}^\ell\rangle.
		\end{aligned}
	\end{equation}
	Furthermore, if  $\FQ = \text{\normalfont poly}_\FQ(\FA):=\sum_{k=0}^h q_k \FA^k$, then
	\begin{equation}\label{equ:cost-seperation-eigenvalues}
		\begin{aligned}
			\langle \Fx_t, \FQ \Fx_t \rangle =   q_0 \|\breve{\Fx}_t\|_2^2 +  \sum_{\ell=1}^d \text{\normalfont poly}_\FQ(\lambda_\ell) \|\bar{\Fx}_t^\ell\|_2^2.
		\end{aligned}
	\end{equation}
\end{lemma}
\vspace{-0.3cm}
\begin{proof}
	See Appendix \ref{apx:proof-lemma-1}.
\end{proof}

\begin{lemma} \label{lem:decoupling-in-cost}
If \textbf{Assumptions \ref{ass:graphon-team-3}} and \textbf{\ref{ass:graphon-team-4}}  hold, then the instantaneous cost and the terminal cost can be decoupled as follows:
\vspace{-0.3cm}
\begin{align*}
	&c_t(\Fu_t,\Fx_t)=\breve{c}_t(\breve{\Fu}_t,\breve{\Fx}_t)+ \sum_{\ell=1}^d \bar{c}^\ell_t(\bar{\Fu}^\ell_t,\bar{\Fx}^\ell_t),\\
&c_T(\Fx_t)=\breve{c}_T(\breve{\Fx}_t)+ \sum_{\ell=1}^d \bar{c}^\ell_T(\bar{\Fx}^\ell_t),
\end{align*}
 where $$
\begin{aligned}
	& \bar{c}^\ell_t(\bar{\Fu}^\ell_t,\bar{\Fx}^\ell_t)= \text{\normalfont poly}_\FQ(\lambda_\ell) \|\bar{\Fx}_t^\ell\|_2^2+ \|\bar{\Fu}^\ell_t\|_2^2, \\[0.5em]
	& \breve{c}_t(\breve{\Fu}_t,\breve{\Fx}_t)=  q_0\|{\breve{\Fx}}_t\|_2^2 +\|\breve{\Fu}_t\|_2^2,\\[0.5em]
	& \bar{c}^\ell_T(\bar{\Fx}^\ell_T) = \text{\normalfont poly}_{\FP_0}(\lambda_\ell)  \|\bar{\Fx}^\ell_T\|_2^2,  \quad \text{ and } \quad  \breve{c}_T(\breve{\Fx}_T) = z_0 \|\breve{\Fx}_T\|_2^2.
\end{aligned} $$
~
\end{lemma}
\vspace{-0.3cm}
\begin{proof}
	By applying the result in Lemma \ref{lem:seperation-in-quadratic-functions} to the cost functions, we obtain the result. 
\end{proof}

\subsection{Decoupled LQR Problems}

Based on Lemma \ref{lem:decoupling-dynamics} and Lemma \ref{lem:decoupling-in-cost}, 
we can separate the LQR  problem into $(d+1)$ decoupled LQR problems:
\begin{enumerate}
	\item the eigensystem LQR problems
\begin{equation}\label{equ: eigen-direction-dyanmics}
	\begin{cases}
		\dot{\bar{\Fx}}_t^\ell = (\alpha_0 + \lambda_\ell) \bar{\Fx}^\ell_t + \text{\normalfont poly}_{\FB}(\lambda_\ell) \bar{\Fu}_t^\ell,\\[.5em]	
			\bar{J}^\ell(\bar{\Fu}^\ell) =\int_0^T \bar{c}^\ell_t(\bar{\Fu}^\ell_t,\bar{\Fx}^\ell_t) dt + \bar{c}^\ell_T(\bar{\Fx}^\ell_T), 1\leq l \leq d\\
	\end{cases}
\end{equation}
where $\bar{c}^\ell_t(\bar{\Fu}^\ell_t,\bar{\Fx}^\ell_t)= \text{\normalfont poly}_\FQ(\lambda_\ell) \|\bar{\Fx}_t^\ell\|_2^2+ \|\bar{\Fu}^\ell_t\|_2^2$ and $\bar{c}^\ell_T(\bar{\Fx}^\ell_T) = \text{\normalfont poly}_{\FP_0}(\lambda_\ell)  \|\bar{\Fx}^\ell_T\|_2^2$;\\
\item the auxiliary system LQR problem
\begin{equation} \label{equ: auxiliary-dynamics}
	\begin{cases}
		 \dot{\breve{\Fx}}_t= \alpha_0 \breve{\Fx}_t+ \beta_0 \breve{\Fu}_t,\\[.5em]
		 \breve{J}(\breve{\Fu}) =\int_0^T \breve{c}_t(\breve{\Fu}_t,\breve{\Fx}_t) dt + \breve{c}_T(\breve{\Fx}_T),\\
	\end{cases}
\end{equation}
where $\breve{c}_t(\breve{\Fu}_t,\breve{\Fx}_t)=  q_0\|{\breve{\Fx}}_t\|_2^2 +\|\breve{\Fu}_t\|_2^2$ and $\breve{c}_T(\breve{\Fx}_T) = z_0 \|\breve{\Fx}_T\|_2^2$.
\end{enumerate}

\begin{lemma} \label{lem:centralized-decoupled-problem-equivalence}
If \textbf{Assumptions \ref{ass:graphon-team-2}-\ref{ass:graphon-team-4}} are satisfied, 
then solving the optimal control problems \eqref{equ: eigen-direction-dyanmics} and \eqref{equ: auxiliary-dynamics} is equivalent to solving the original optimal control problem defined by \eqref{equ:infinite-system-model} and \eqref{equ:original-cost-function}. Moreover, the optimal control solution exists and is unique. 
\end{lemma}
\begin{proof}

Firstly, the original dynamics defined by \eqref{equ:infinite-system-model} are decoupled into dynamics of the auxiliary system and those of eigensystems. 
Secondly, the cost defined by \eqref{equ:original-cost-function} can be decoupled as 
%\begin{equation}
	$J(\Fu) = \breve{J}(\breve{\Fu})+ \sum_{\ell=1}^d \bar{J}^\ell (\bar{\Fu}^\ell),$
%\end{equation}
with the summation of non-negative terms on the right hand side. Therefore $J(\Fu)$ is minimized if and only if  $\breve{J}(\breve{\Fu})$ and $\bar{J}^\ell(\bar{\Fu}^\ell), (1\leq l \leq d)$, are minimized. Hence, solving the optimal control problems \eqref{equ: eigen-direction-dyanmics} and \eqref{equ: auxiliary-dynamics} is equivalent to solving the original optimal control problem defined by \eqref{equ:infinite-system-model} and \eqref{equ:original-cost-function}. 
The existence and uniqueness of the optimal solution follow  Proposition \ref{prop:existence-unique-opt-sol}.
\end{proof}

\subsection{Centralized Optimal Solution}
\begin{theorem}
If \textbf{Assumptions \ref{ass:graphon-team-2}-\ref{ass:graphon-team-4}} are satisfied, 
then the  optimal control law for the optimal control problem in Section \ref{sec:ocp} is given by 
\begin{equation} \label{equ:centralized-optimal-control}
	\Fu_t= - \beta_0 \FL_{T-t} \breve{\Fx}_t -  \sum_{\ell=1}^d\text{\normalfont poly}_{\FB}(\lambda_\ell) {\FM}_{T-t}^{\ell}\Fx^\ell_t ,
\end{equation}
where $\FL:=\{\FL_t: t\in[0,T]\}$ is the solution to the Riccati equation 
\begin{equation} \label{equ: Riccati-Equation-Auxillary-System}
	\dot{\FL}_t = 2 \alpha_0 \FL_t - \beta^2_0 \FL_t^2 + q_0 \I , \quad \FL_0=z_0\I,
\end{equation}
and 
$\FM^{\ell} := \{\FM^{\ell}_t: t\in[0,T] \}$ is the solution to the Riccati equation 
\begin{equation}\label{equ: Riccati-Equation-Eigendirection}
	\begin{aligned}
		&\dot{\FM}^{\ell}_t = 2(\alpha_0+\lambda_\ell){\FM}^{\ell}_t- \text{\normalfont poly}_{\FB}(\lambda_\ell)^2({\FM}^{\ell}_t)^2+\text{\normalfont poly}_\FQ(\lambda_\ell)\I,\\[0.5em]
		 &{\FM}^{\ell}_0= \text{\normalfont poly}_{\FP_0}(\lambda_\ell)\I , \quad 1 \leq l \leq d .
	\end{aligned}
\end{equation}
\end{theorem}
\begin{proof}
	Since the control problems for the auxiliary system and eigensystems are decoupled, one can solve these problems independently based on the LQR controls for infinite dimensional system \cite{bensoussan2007representation}.  The optimal control laws are given by $\breve{\Fu}_t=- \beta_0 \FL_{T-t} \breve{\Fx}_t$ and ${\Fu}^\ell_t  = - \text{\normalfont poly}_{\FB}(\lambda_\ell)  {\FM}_{T-t}^{\ell} {\Fx}^\ell_t $, respectively, where $\FL:=\{\FL_t: t\in[0,T]\}$ is the solution to the Riccati equation \eqref{equ: Riccati-Equation-Auxillary-System} and $\FM^{\ell} := \{\FM^{\ell}_t: t\in[0,T] \}$ is the solution to the Riccati equation \eqref{equ: Riccati-Equation-Eigendirection}.
Furthermore, since $\Fu_t= \breve{\Fu}_t + \sum_{\ell=1}^d \bar{\Fu}^\ell_t $, we obtain \eqref{equ:centralized-optimal-control}.
\end{proof}

\subsection{Localized Optimal Solutions}
%We use ``eigenfunction direction'' to represent the space spanned by the corresponding eigenfunction.

To obtain the optimal solution in a localized manner, each subsystem should solve the following optimal control problems  in all eigenfunction directions: %in the $i^{th}$ eigenfunction direction,
\begin{equation} \label{equ:local-eigen-direction-system}
	\begin{cases}
		\dot{\bar{x}}_t^\ell = (\alpha_0 + \lambda_\ell) \bar{x}^\ell_t + \text{\normalfont poly}_{\FB}(\lambda_\ell) \bar{u}_t^\ell,\\[.5em]
		\bar{J}^\ell(\bar{u}^\ell_t)= \int_0^T \bar{c}^\ell_t(\bar{u}^\ell_t,\bar{x}^\ell_t) dt + \bar{c}^\ell_T(\bar{x}^\ell_T), 1\leq \ell \leq d, \\
		  %\langle\bar{\Fu}^i_t,\bar{\Fu}^i_t\rangle \\
%	
			  % \langle \bar{\Fx}^i_T, \FP_0 \bar{\Fx}^i_T \rangle, %\quad \FP_0\geq 0. 
%
	\end{cases}
\end{equation}
where  $\bar{x}_t^\ell$ and $ \bar{u}_t^\ell$ are the scalar values given in \eqref{eq:scalar-inner-product}, 
$$
\begin{aligned}
	&\bar{c}^\ell_t(\bar{u}^\ell_t,\bar{x}^\ell_t)= \text{\normalfont poly}_\FQ(\lambda_\ell) (\bar{x}_t^\ell)^2+ (\bar{u}^\ell_t)^2,\\
	&\bar{c}^\ell_T(\bar{x}^\ell_T) = \text{\normalfont poly}_{\FP_0}(\lambda_\ell)  (\bar{x}^\ell_T)^2.
\end{aligned}
$$
In addition, for the subsystem with the index $\gamma \in [\underline{\gamma}, \overline{\gamma}]\subset [0,1]$ where $\underline{\gamma}$ and $\overline{\gamma}$ are respectively the lower bound and the upper bound for the interval corresponding to subsystem $\gamma$, it should solve the following optimal control problem of the auxiliary system: 
\begin{equation} \label{equ:local-auxiliary-system}
	\begin{cases}
		 \dot{\breve{\Fx}}_t(\gamma)= \alpha_0 \breve{\Fx}_t(\gamma)+ \beta_0 \breve{\Fu}_t(\gamma),\\[.5em]
		 \TJ(\breve{\Fu}(\gamma))= \int_0^T \breve{c}_t(\breve{\Fu}_t(\gamma),\breve{\Fx}_t(\gamma)) dt + \breve{c}_T(\breve{\Fx}_T(\gamma)), \\
	\end{cases}
\end{equation}
where $$
\begin{aligned}
&\breve{c}_t\left(\breve{\Fu}_t(\gamma\right),\breve{\Fx}_t(\gamma))= q_0({\breve{\Fx}}_t(\gamma))^2 +(\breve{\Fu}_t(\gamma))^2,\\
&\breve{c}_T(\breve{\Fx}_T(\gamma)) = z_0 (\breve{\Fx}_T(\gamma))^2 .	
\end{aligned}
$$
\begin{theorem} \label{thm:team-optimal-control}
If \textbf{Assumptions \ref{ass:graphon-team-2}-\ref{ass:graphon-team-4}} are satisfied, 
then solving the optimal control problems defined by \eqref{equ:local-eigen-direction-system} and \eqref{equ:local-auxiliary-system} locally is equivalent to solving the original optimal control problem defined by \eqref{equ:infinite-system-model} and \eqref{equ:original-cost-function}.
Moreover, the localized optimal control law for the $\gamma^{th}$ subsystem with $\gamma \in [\underline{\gamma}, \overline{\gamma}]\subset [0,1]$ is given by 
\begin{equation}\label{equ:team-optimal-control-law}
	\Fu_t(\gamma)= - \beta_0 L_{T-t} \breve{\Fx}_t(\gamma) -   \sum_{\ell=1}^d \text{\normalfont poly}_{\FB}(\lambda_\ell) {M}_{T-t}^{\ell} \bar{x}^\ell_t \Ff_\ell(\gamma),
\end{equation}
where $L:=\{L_t: t\in[0,T]\}$ is the solution to the scalar Riccati equation 
\begin{equation} \label{equ:team-Ricc-aux}
	\dot{L}_t = 2 \alpha_0 L_t - \beta^2_0 L_t^2 + q_0, \quad L_0=z_0,
\end{equation}
and 
$M^{\ell} := \{M^{\ell}_t: t\in[0,T] \}$ is the solution to the scalar Riccati equation 
\begin{equation}\label{equ:team-Ricc-eigen}
	\begin{aligned}
		& \dot{M}^{\ell}_t = 2(\alpha_0+\lambda_\ell){M}^{\ell}_t- \text{\normalfont poly}_{\FB}(\lambda_\ell)^2({M}^{\ell}_t)^2+\text{\normalfont poly}_\FQ(\lambda_\ell),\\
		& {M}^{\ell}_0= \text{\normalfont poly}_{\FP_0}(\lambda_\ell) , \quad 1 \leq l \leq d .
	\end{aligned}
\end{equation}

\end{theorem}

\begin{proof}
First, since $\bar{\Fx}_t^\ell = \bar{x}_t^\ell \Ff_\ell$, $\bar{\Fu}_t^\ell = \bar{u}_t^\ell \Ff_\ell$ and $\|\Ff_\ell\|_2 =1$ for $1\leq l \leq d$, 
 the optimal control problem in \eqref{equ: eigen-direction-dyanmics} can be equivalently solved by solving 
\eqref{equ:local-eigen-direction-system} and then recovering the pair $(\bar{\Fx}_t^\ell, \bar{\Fu}_t^\ell)$ in the $\Ff_\ell$ eigendirection, $1\leq l \leq d$. 
Second, notice that $J(\breve{\Fu}) = \int_0^1 J(\breve{\Fu}_t(\gamma)) d\gamma$ and $J(\breve{\Fu}(\gamma))$ are non-negative for any $\gamma \in [\underline{\gamma}, \overline{\gamma}]\subset[0,1]$. 
Therefore the optimal control problems \eqref{equ: auxiliary-dynamics} and  \eqref{equ:local-auxiliary-system} are equivalent.  
These, together with the result in Lemma \ref{lem:centralized-decoupled-problem-equivalence}, imply that solving the optimal control problems defined by \eqref{equ:local-eigen-direction-system} and \eqref{equ:local-auxiliary-system} locally is equivalent to solving the original optimal control problem defined by \eqref{equ:infinite-system-model} and \eqref{equ:original-cost-function}.

It is obvious that \eqref{equ:team-Ricc-aux} and \eqref{equ:team-Ricc-eigen} are the Riccati equations for the LQR problems  in \eqref{equ:local-auxiliary-system} and \eqref{equ:local-eigen-direction-system}, respectively.  Based on the standard LQR theory, the optimal control laws are respectively given by 
	\begin{equation*}
	\begin{aligned}
		&\breve{\Fu}_t(\gamma)  =- \beta_0 L_{T-t} \breve{\Fx}_t(\gamma),
&\bar{u}^\ell_t  = - \text{\normalfont poly}_{\FB}(\lambda_\ell)  {M}_{T-t}^{\ell} \bar{x}^\ell_t.\label{eq:l-dir-ctrl}
	\end{aligned}
	\end{equation*}	
Furthermore, since $\Fu_t(\gamma)= \breve{\Fu}_t(\gamma) + \sum_{\ell=1}^d \bar{\Fu}^\ell_t (\gamma)= \breve{\Fu}_t(\gamma) + \sum_{\ell=1}^d \bar{u}^\ell_t \Ff_\ell(\gamma)$, we obtain the localized optimal control law in \eqref{equ:team-optimal-control-law} for the original problem defined by \eqref{equ:infinite-system-model} and \eqref{equ:original-cost-function}. 
\end{proof}

The optimal control \eqref{equ:team-optimal-control-law} consists of a single auxiliary component and $d$ eigendirection components. The eigenstates $x^\ell_t=\langle \Fx, \Ff_\ell \rangle, 1\leq \ell \leq d,$ may be viewed as the global weighted aggregates of states;  the auxiliary state $\breve \Fx_t(\gamma)= \Fx_t(\gamma) - \sum_{\ell=1}^d x_t^\ell \Ff_\ell(\gamma)$ may be viewed as the local state offset from the global state aggregates. % 

\subsection{Information Structure and Complexity}

The following information is required by a representative subsystem $\gamma \in [\underline{\gamma}, \overline{\gamma}]\subset [0,1]$ to generate the localized optimal solution: 
\begin{enumerate}
	\item all the eigenvalues of $\FA$ and the value of the respective eigenfunctions at its index location,  that is, $\lambda_\ell$,  $\Ff_\ell(\gamma)$ for all $1\leq \ell \leq d$;
	%eigenfunction and its own location for each eigen-direction
	\item 
	the projections of the state $\Fx_t$ onto each eigenfunction direction, that is, $\bar{x}^\ell_t=\langle \Fx_t, \Ff_\ell\rangle$ for all $1\leq \ell \leq d$;
	\item its own state $\Fx_t(\gamma)$.
\end{enumerate}
Alternatively, 2) can be replaced by the projections of the initial state $\Fx_0$ onto each eigenfunction direction, that is, $\bar{x}^\ell_0=\langle \Fx_0, \Ff_\ell\rangle$ for all $1\leq \ell \leq d$. Given $\bar{x}^\ell_0$, each subsystem can locally precompute the state  $\{\bar{x}^\ell_t, t\in(0,T]\}$ 
 based on the dynamics in \eqref{equ:local-eigen-direction-system} and the optimal control law in the $\Ff_\ell$ eigendirection given by \eqref{eq:l-dir-ctrl}.

The complexity of generating the optimal control law for each subsystem involves solving the scalar Riccati equation corresponding to auxiliary state dynamics and solving $d$ number of scalar Riccati equations corresponding to $d$ eigenfunction directions.

It is worth mentioning that although in the graphon dynamical system there are in general an infinite number of subsystems, each subsystem can still generate the localized optimal solution by solving $(d+1)$ scalar Riccati equations.
If the underlying graphon is an uniform graphon $\FA(x,y) =1$ for all $x,y \in [0,1]$, which gives $d=1$, $\Ff_1= \mathbf{1} \in L^2_{[0,1]}$ and $\lambda_1 =1$, then the LQR problem  with graphon coupling reduces to the LQR problem with mean-field coupling. The corresponding solution  involves solving only two decoupled Riccati equations. 

\section{Suboptimal Solution via Spectral Approximations} \label{sec:graphon-approx}
If \textbf{Assumption \ref{ass:graphon-team-3}} does not hold, that is, $\FA$ contains an infinite number of eigenfunctions corresponding to the non-zero eigenvalues, then one needs to find the approximate solution.
Since for a graphon $\FA \in \GSOT$, we have $\|\FA\|_2 < \infty$ and  hence the operator $\FA$ is a compact operator according to \cite[Chapter 2, Proposition 4.7]{conway2013course}.  Therefore it has a countable spectral decomposition
\begin{equation} \label{equ: graphon-infinite-spectral-sum}
	\FA(x,y) = \sum_{i=1}^{\infty} \lambda_\ell \Ff_\ell(x) \Ff_\ell(y), \quad (x,y)\in[0,1]^2,
\end{equation}
where the convergence is in the $L^2_{[0,1]^2}$ sense, $\{\lambda_1, \lambda_2,....\}$  is the set of eigenvalues (which are not necessarily distinct) with decreasing absolute values, and $\{\Ff_1, \Ff_2,...\}$ represents the set of the corresponding orthonormal eigenfunctions (i.e. $\|\Ff_\ell\|_2=1$, and $\langle \Ff_\ell, \Ff_k\rangle =0$ if $l\neq k$). 
The only accumulation point of the eigenvalues is zero \cite{lovasz2012large}, that is, $\lim_{\ell\rightarrow \infty} \lambda_\ell =0.$ %Every non-zero eigenvalue has finite multiplicity, here $\lambda_\ell$ (including multiplicity) 
This implies that the compact operator $\FA$  can be approximated by a finite truncation of the spectral decomposition with the most significant eigenvalues. %

Since the centralized solution and the localized optimal solution are essentially the same, we focus only on the approximation result for the localized optimal solution.  
Consider a graphon $\FA$ with the spectral decomposition in \eqref{equ: graphon-infinite-spectral-sum} and we approximate it by
\begin{equation}\label{equ:approximation-graphon}
	\FA_L(x,y):=\sum_{\ell=1}^L \lambda_\ell \Ff_\ell(x) \Ff_\ell(y), \quad \forall (x,y) \in [0,1]^2.
\end{equation}
Since $\FA \in L^2_{[0,1]^2}$, it is obvious that $\FA_L \in L^2_{[0,1]^2}$ is a bounded operator and $\lim_{L\rightarrow \infty} \FA_L = \FA.$ %Thus the corresponding infinite dimensional system is well-defined.  
The corresponding auxiliary state and the auxiliary control are given by
\begin{equation}
	\breve{\Fx}_t= \Fx_t -\sum_{\ell=1}^{L} \bar{\Fx}_t^\ell \quad \text{and} \quad \breve{\Fu}_t= \Fu_t -\sum_{\ell=1}^{L} \bar{\Fu}_t^\ell.
\end{equation}

Consider implementing  the localized optimal control law in \eqref{equ:team-optimal-control-law}  with the approximation $\FA_L$ of $\FA$. 
In this implementation, any eigendirection corresponding to $h>L$ is ignored in the spectral approximation and the control law applied in the $h^{th}$  eigensystem $(h>L)$ is given by the auxiliary control law.  
  Now the optimal feedback gain $M^{h}:=\{M^h_t, t\in [0,T]\}$ in any eigendirection corresponding to $h>L$ is based on the following scalar Riccati equation
\begin{equation*}
	\begin{aligned}
		& \dot{M}^{h}_t = 2(\alpha_0+\lambda_h){M}^{h}_t- \text{\normalfont poly}_{\FB}(\lambda_h)^2({M}^{h}_t)^2+\text{\normalfont poly}_\FQ(\lambda_h),\\
		& {M}^{h}_0= \text{\normalfont poly}_{\FP_0}(\lambda_h)
	\end{aligned}
\end{equation*}
In the approximation scheme, this is reproduced by the feedback gain $\widetilde M^{h}:=\{\widetilde M^{h}_t, t\in[0,T]\}$ given by the following scalar Riccati equation
\begin{equation}
	\begin{aligned}
		& \dot{\widetilde M^h_t} = 2\alpha_0{\widetilde M^{h}_t}- (\beta_0\widetilde M^{h}_t)^2+q_0,\quad\widetilde M^{h}_0 = z_0,
	\end{aligned}
\end{equation}
which is based on \eqref{equ:team-Ricc-aux}.

For simplicity of discussion, let $\text{poly}_\FB(\FA)= \beta_0 \I$.
Then the closed loop system in the $h^{th}$ eigendirection $(\text{for any } h>L)$  under the optimal control is given by 
\begin{equation}
		\dot{\bar{x}}_t^h = (\alpha_0 + \lambda_h - \beta_0^2 M_{T-t}^h) \bar{x}^h_t 
\end{equation}
with the solution 
	$\bar{x}^h_t= \text{exp} \Big( \int_0^T (\alpha_0 + \lambda_h - \beta_0^2 M_{T-t}^h)dt\Big) \bar{x}^h_0.$

If all subsystems implement the observation of eigenstates $\{x^\ell\}_{\ell=1}^L$,
 then the closed loop system in the $h^{th}$ eigendirection  under the approximate control is given by 
\begin{equation}
		\dot{\widetilde{x}}_t^h = (\alpha_0 + \lambda_h - \beta_0^2 \widetilde M^{h}_{T-t}) \widetilde{x}^h_t 
\end{equation}
with the solution 
	$\widetilde{x}^h_t= \text{exp} \Big( \int_0^T (\alpha_0 + \lambda_h - \beta_0^2 \widetilde M_{T-t}^h)dt\Big) \widetilde{x}^h_0.$

Note that ${x}^h_0= \widetilde{x}^h_0$. Therefore, 
\[
	\frac{\widetilde x_t^h}{\bar x_t^h} = \text{exp} \Big(- \beta_0^2\int_0^T (\widetilde M_{t}^h-  M_{t}^h)dt \Big). 
\]
This, together with the explicit solutions to Riccati equations (see Appendix \ref{apx:riccati-explicit-sol}), leads to the following proposition.

\begin{proposition}
	Assume 
	$\textup{poly}_\FB(\FA)= \beta_0 \I$. If the localized optimal control law \eqref{equ:team-optimal-control-law} is applied with the approximation of $\FA$ by $\FA_L$ given in \eqref{equ:approximation-graphon} and all subsystems implement the real-time observation of eigenstates $\{x^\ell\}_{\ell=1}^L,$
	then 
\[
	\frac{\widetilde x_t^h}{\bar x_t^h} = \text{\normalfont exp} \Big(- \beta_0^2\int_0^T (\widetilde M_{t}^h-  M_{t}^h)dt \Big),\quad \forall h>L. 
\]
Furthermore, if $\widetilde M_{t}^h-\widetilde P^{h}\neq 0$ and $ M_{t}^h- P^{h}\neq 0$ where $\widetilde P^{h}= \sqrt{\frac{\alpha_0^2}{\beta_0^4}+ \frac{q_0}{\beta_0^2} } + \frac{\alpha_0}{\beta_0^2}$ and $P^{h}=\sqrt{\frac{(\alpha_0+\lambda_h)^2}{\beta_0^4}+ \frac{\text{\normalfont poly}_\FQ(\lambda_h)}{\beta_0^2}  } + \frac{(\alpha_0+ \lambda_h)}{\beta_0^2}$,
 then $\widetilde M_{t}^h$ and $ M_{t}^h$ are explicitly given by 
 \begin{equation*}
\begin{aligned}
	\widetilde M_{t}^h = \left[\frac{e^{^{-2(\alpha_0 - \beta_0^2\widetilde P^h)t}}}{z_0-\widetilde P^h}+ \beta_0^2\int_0^t e^{^{-2(\alpha_0 - \beta_0^2\widetilde P^h)\tau}}d\tau\right]^{^{-1}}+ \widetilde P^h,
\end{aligned}
\end{equation*}
and 
\begin{equation*}
\begin{aligned}
	 &M_{t}^h =\\
	  & \left[\frac{e^{^{-2(\alpha_0 +\lambda_h- \beta_0^2 P^h)t}}}{{\text{\normalfont poly}_{_{\FP_0}}(\lambda_h)- P^h}}+ \beta_0^2 \int_0^t  e^{^{-2(\alpha_0 +\lambda_h- \beta_0^2 P^h)\tau}} d\tau\right]^{^{-1}}+  P^h.\\
\end{aligned}
\end{equation*}
\end{proposition}

\section{Discussion}
We limit our discussions to graphons $\FA \in \GSOT$ purely for simplicity.  The generalization to functions $\FA_c\in \GST_c$ is immediate.
Following the solution approach, the generalization will only result in the difference in the magnitude of eigenvalues in the spectral decomposition. 
Note that the corresponding $L^2[0,1]$ operator generated by $\FA_c \in \GST_c$ is a compact operator \cite[Chapter 2, Proposition 4.7]{conway2013course} and the approximation \eqref{equ:approximation-graphon} for the suboptimal solution also holds.

The idea of decoupling in generating the optimal control law is inspired by \cite{arabneydi2015team,arabneydi2016team}. The coupling in this paper takes into account the local network weights and hence is more general than (weighted or unweighted) mean-field coupling in \cite{arabneydi2015team,arabneydi2016team}, and  the spectral decomposition of graphons is further required in the decoupling in this paper.

%For simplicity we limit our discussions to graphons $\FA_1: [0,1]^2 \rightarrow [-1,1]$, the generalization to bounded symmetric measurable functions $\FA_c: [0,1]^2 \rightarrow [-c,c]$ with $c\in \BR$ is immediate. 
\section{Example}\label{sec:numerical-example}
Consider the example with the following parameters: 
$\alpha_0 = 2$, $\text{poly}_\FB (s) = 1+\frac12 s$, $\text{poly}_\FQ(s)= (1-s)^2$, $\text{poly}_{\FP_0}(s)= (1-s)^2$. Consider a sinusoidal graphon $\FA$ given by
\begin{equation}\label{eq:sinusoidal-graphon}
	\FA (x,y) = \cos(2\pi(x-y)),\quad  \forall (x,y) \in [0,1]^2.
\end{equation}
Note that $\FA$ has two eigenfunctions $\Ff_1=\sqrt{2}\sin 2\pi(\cdot)$ and $\Ff_2=\sqrt{2}\cos 2\pi(\cdot)$ corresponding to the only nonzero eigenvalue $ \lambda_1=\lambda_2=\frac12$. Evidently, \textbf{Assumptions 1-4} are satisfied. 

The auxiliary system and the auxiliary cost for subsystem $\gamma \in [\underline{\gamma}, \overline{\gamma}]\subset [0,1]$~are respectively given by 
\begin{align*}
	&\dot{\breve \Fx}_t(\gamma) = 2\breve \Fx_t(\gamma) + \breve \Fu_t(\gamma),\\
	&\breve J(\Fu) = \int_0^T \big(\Fx^2_t(\gamma)+ \Fu^2_t(\gamma)\big) dt + \Fx^2_T(\gamma).	
\end{align*}
The dynamics and cost for the $\ell^{th}$ eigensystem, $\ell \in \{1,2\}$, are respectively given by
\begin{align*}
	&\dot{\bar x}_t^\ell = \frac52\bar x_t^\ell +   \frac54 \bar u_t^\ell,\\
	&\bar{J}^\ell (\bar u^\ell) = \int_0^T \big(\frac14  (\bar x_t^\ell)^2 + (\bar u_t^\ell)^2\big) dt + \frac14 (\bar x_T^\ell)^2 . 
\end{align*}
Following Theorem \ref{thm:team-optimal-control}, the localized optimal control problem for the system \eqref{equ:infinite-system-model} with the cost \eqref{equ:original-cost-function} for a subsystem $\gamma \in [\underline{\gamma}, \overline{\gamma}]\subset [0,1]$ is given as follows: 
\begin{multline}\label{eq:num-sol}
	\Fu_t(\gamma)= - L_{T-t} \breve{\Fx}_t(\gamma) \\-   \frac{5\sqrt{2}}4 \Big({M}_{T-t}^{(1)} \bar{x}^{(1)}_t \sin2\pi\gamma + {M}_{T-t}^{(2)}\bar{x}^{(2)}_t \cos2\pi\gamma\Big),
\end{multline}
where  
$$\breve{\Fx}_t(\gamma) = {\Fx}_t(\gamma) - \sqrt{2} \bar{x}^{(1)}_t \sin2\pi \gamma -\sqrt{2} \bar{x}^{(2)}_t \cos2\pi \gamma,$$~
\begin{equation*}
\begin{aligned}
	&\dot{L}_t = 4 L_t -  L_t^2 + 1, \quad L_0=1, \\
	& \dot{M}^{\ell}_t = 5{M}^{\ell}_t- (\frac54 {M}^{\ell}_t)^2+\frac14,\quad
		 {M}^{\ell}_0= \frac14 , \quad \ell \in\{1,2\},
\end{aligned}
\end{equation*}
with $t\in[0,T]$. 

\begin{figure}[htb]
\centering
	\includegraphics[width=8.5cm]{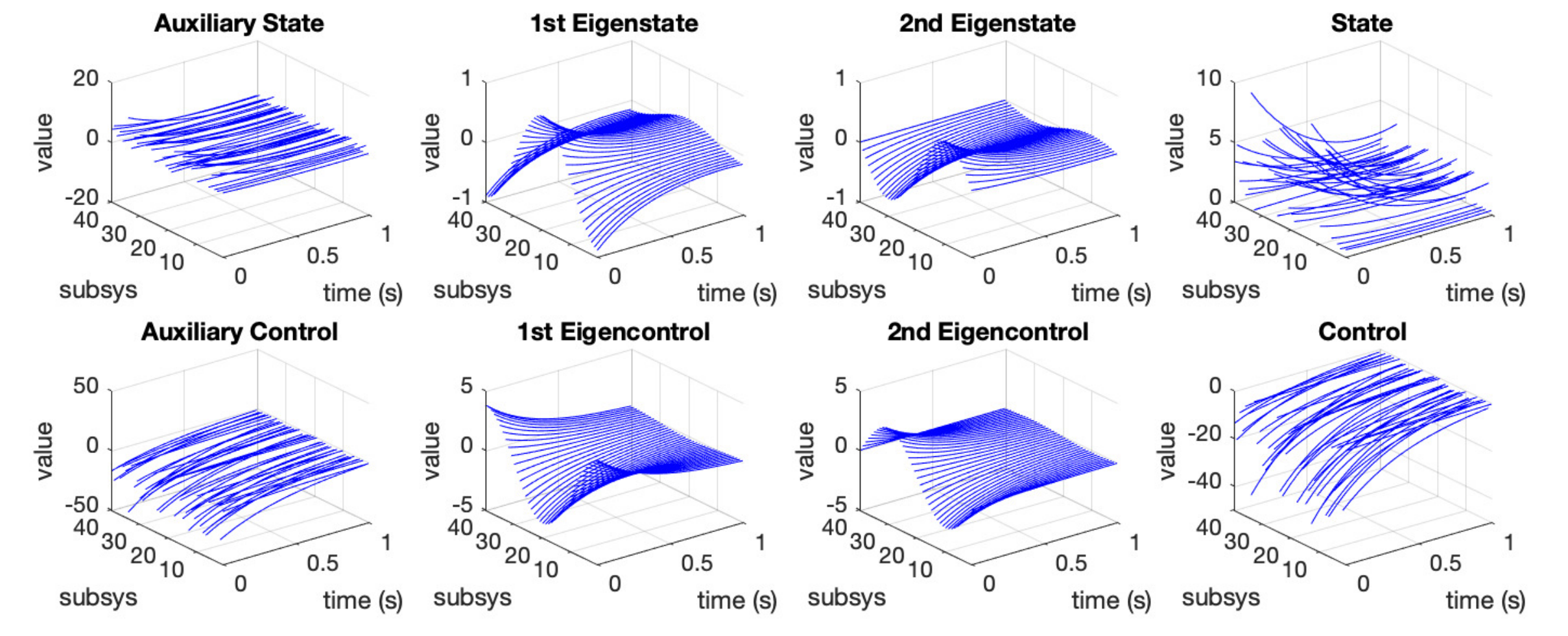}
	\caption{This is a simulation demonstration for the example in Section \ref{sec:numerical-example}. The simulation runs on the step function system corresponding to the sinusoidal graphon \eqref{eq:sinusoidal-graphon} based on the uniform partition of size 40. 
	  Note that the step function system represents a network system consisting of 40 nodal subsystems where each subsystem is indexed by an interval of length $\frac1{40}$ in $[0,1]$.
	Each subsystem locally generates it control input according to \eqref{eq:num-sol}, which requires solving only two scalar Riccati equations. The initial states are generated randomly.}  % instead of a $40\times 40$-dimensional matrix Riccati equation
\end{figure}

\section{Acknowledgment}
The authors would like to thank Prof. Aditya Mahajan and the reviewers for constructive comments and suggestions. 

\bibliographystyle{IEEEtran}
\bibliography{/Users/smartgao/Dropbox/ImportantBib/IEEEabrv,/Users/smartgao/Dropbox/ImportantBib/mybib-abrv}

\appendices
\section{Proof of Lemma \ref{lem:seperation-in-quadratic-functions}} \label{apx:proof-lemma-1}

\begin{proof}
	First, we show for any $k\geq0 $
	\begin{equation} \label{equ:decoupling-cost-power-function}
				\begin{aligned}
						\langle \Fx_t, \FA^k \Fx_t \rangle 
						=&\langle  \breve{\Fx}_t, \FA^k \breve{\Fx}_t \rangle+  \sum_{\ell=1}^d \langle  {\bar{\Fx}}^\ell_t, \FA^k \bar{\Fx}^\ell_t\rangle.	\\
		% \langle \Fx_t, \FA^k \Fx_t \rangle 				=& \langle  \breve{\Fx}_t, \FA^k \breve{\Fx}_t \rangle + \langle \sum_{i=1}^d \bar{x}_t^i, \FA^k \sum_{i=1}^d \bar{x}_t^i\rangle\\
						% & + 2 \langle \breve{x}_t, \FA^k \sum_{i=1}^d \bar{x}_t^i \rangle		\\=&\langle  \breve{\Fx}_t, \FA^k \breve{\Fx}_t \rangle+  \sum_{i=1}^d \langle  {\bar{x}}^i_t, \FA^k \bar{x}^i\rangle
				\end{aligned}
	\end{equation}

By decomposing the left hand side of \eqref{equ:decoupling-cost-power-function}, we have
		\begin{equation} \label{equ:decoupling-proof-seperation}
					\begin{aligned}
							% \langle \Fx_t, \FA^k \Fx_t \rangle
							% =&\langle  \breve{\Fx}_t, \FA^k \breve{\Fx}_t \rangle+  \sum_{i=1}^d \langle  {\bar{x}}^i_t, \FA^k \bar{x}^i\rangle	\\
			\langle \Fx_t, \FA^k \Fx_t \rangle 				=& \langle  \breve{\Fx}_t, \FA^k \breve{\Fx}_t \rangle + \langle \sum_{\ell=1}^d \bar{\Fx}_t^\ell, \FA^k \sum_{\ell=1}^d \bar{\Fx}_t^\ell\rangle\\
							&+ 2 \langle \breve{\Fx}_t, \FA^k \sum_{\ell=1}^d \bar{\Fx}_t^\ell \rangle.		\\
							% =&\langle  \breve{\Fx}_t, \FA^k \breve{\Fx}_t \rangle+  \sum_{i=1}^d \langle  {\bar{x}}^i_t, \FA^k \bar{x}^i\rangle
					\end{aligned}
		\end{equation}

Further, the second term on the right hand side of \eqref{equ:decoupling-proof-seperation} gives  
\begin{equation*}
	\begin{aligned}
		&\left\langle \sum_{\ell=1}^d \bar{\Fx}_t^\ell, \FA^k \sum_{\ell=1}^d \bar{\Fx}_t^\ell\right\rangle =
		\left\langle \sum_{\ell=1}^d \langle{\Fx}_t, \Ff_\ell\rangle \Ff_\ell, \FA^k \sum_{\ell=1}^d \langle{\Fx}_t, \Ff_\ell\rangle \Ff_\ell \right\rangle\\
%		& = \langle \sum_{\ell=1}^d \langle{\Fx}_t, \Ff_\ell\rangle \Ff_\ell, \lambda_\ell^k \sum_{\ell=1}^d \langle{\Fx}_t, \Ff_\ell\rangle \Ff_\ell \rangle\\
		& = \sum_{\ell=1}^d \left\langle  \langle{\Fx}_t, \Ff_\ell\rangle \Ff_\ell,\lambda_\ell^k \langle{\Fx}_t, \Ff_\ell\rangle \Ff_\ell \right\rangle 
%		& = \sum_{\ell=1}^d \left\langle  \langle{\Fx}_t, \Ff_\ell\rangle \Ff_\ell,\FA^k \langle{\Fx}_t, \Ff_\ell\rangle \Ff_\ell \right\rangle \\
		 = \sum_{\ell=1}^d \langle  {\bar{\Fx}}^\ell_t, \FA^k \bar{\Fx}^\ell_t\rangle.	% = \sum_{\ell=1}^d \lambda_\ell^k\|\bar{\Fx}^\ell_t\|_2^2
	\end{aligned}
\end{equation*}
and the last term on the right hand side of \eqref{equ:decoupling-proof-seperation} gives
\begin{equation*}
	\begin{aligned}
		&\left\langle \breve{\Fx}_t, \FA^k \sum_{\ell=1}^d \bar{\Fx}_t^\ell\right\rangle =
		\left\langle {\Fx}_t- \sum_{\ell=1}^d \bar{\Fx}_t^\ell,  \FA^k \sum_{\ell=1}^d \bar{\Fx}_t^\ell \right \rangle\\
		& =
		\left\langle {\Fx}_t- \sum_{\ell=1}^d \langle{\Fx}_t, \Ff_\ell\rangle \Ff_\ell,  \FA^k \sum_{\ell=1}^d \langle{\Fx}_t, \Ff_\ell\rangle \Ff_\ell \right \rangle\\
%		 &=\langle {\Fx}_t,  \sum_{\ell=1}^d \lambda_\ell^k \langle{\Fx}_t,  \Ff_\ell\rangle \Ff_\ell \rangle \\
%		 &-\langle \sum_{\ell=1}^d \langle{\Fx}_t, \Ff_\ell\rangle \Ff_\ell,  \sum_{\ell=1}^d \lambda_\ell^k \langle{\Fx}_t, \Ff_\ell\rangle \Ff_\ell \rangle\\
		 &=  \sum_{\ell=1}^d \lambda_\ell^k \langle{\Fx}_t,  \Ff_\ell\rangle^2  -\sum_{\ell=1}^d  \lambda_\ell^k \langle{\Fx}_t, \Ff_\ell\rangle^2 \|\Ff_\ell\|_2^2
%		& \quad (\text{since } \|\Ff_\ell\|_2=1)\\
	 =0 .
	\end{aligned}
\end{equation*}
Hence, we obtain \eqref{equ:decoupling-cost-power-function}.  Since this separation result holds for all powers of $\FA$, we have %together with the linearity of inner product, we have
			\begin{equation*} %\label{equ:decoupling-cost-power-function}
				\begin{aligned}
						\langle \Fx_t, \text{\normalfont poly}(\FA) \Fx_t \rangle 
						=&\langle  \breve{\Fx}_t, \text{\normalfont poly}(\FA) \breve{\Fx}_t \rangle+  \sum_{\ell=1}^d \langle  {\bar{\Fx}}^\ell_t, 
						\text{\normalfont poly}(\FA) \bar{\Fx}^\ell\rangle	\\
					\end{aligned}
			\end{equation*}
With $\FQ=\text{\normalfont poly}(\FA)$, we prove the result in \eqref{equ:cost-seperation}. 
Furthermore, 
\begin{equation}
	\begin{aligned}
		\langle  {\bar{\Fx}}^\ell_t, \text{\normalfont poly}(\FA) \bar{\Fx}^\ell\rangle 
		%&= \Big\langle  \langle \Fx_t, \Ff_\ell \rangle \Ff_\ell, \text{\normalfont poly}(\FA) \langle \Fx_t,\Ff_\ell \rangle \Ff_\ell \Big \rangle\\
		&	= \Big \langle  \langle \Fx_t, \Ff_\ell \rangle \Ff_\ell, \text{\normalfont poly}(\lambda_\ell) \langle \Fx_t,\Ff_\ell  \rangle \Ff_\ell \Big \rangle\\
		& = \text{\normalfont poly}(\lambda_\ell) \|\bar{\Fx}_t^\ell\|_2^2		
	\end{aligned}	
\end{equation}
and 
\begin{equation}
	\begin{aligned}
		\langle  \breve{\Fx}_t, \FQ \breve{\Fx}_t \rangle & = \langle  \breve{\Fx}_t, q_0 \breve{\Fx}_t \rangle + \left\langle  \breve{\Fx}_t, \sum_{k=1}^h q_k\FA^k \breve{\Fx}_t \right\rangle \\
		%& = \langle  \breve{\Fx}_t, q_0 \breve{\Fx}_t \rangle + \sum_{k=1}^h q_k \langle  \breve{\Fx}_t, \FA^k \breve{\Fx}_t \rangle 
%		& \quad (\text{since } \langle \breve{\Fx}_t, \FA \breve{\Fx}_t \rangle =0  )\\
		&= \langle  \breve{\Fx}_t, q_0 \breve{\Fx}_t \rangle = q_0 \|\breve{\Fx}_t\|_2^2.
	\end{aligned}
\end{equation}
Therefore, we have the result in \eqref{equ:cost-seperation-eigenvalues}.
\end{proof}

\section{Explicit Solutions to Scalar Riccati Equations} \label{apx:riccati-explicit-sol}
Consider the following scalar Riccati equation:
\begin{equation}\label{equ:differential-Ricc}
	\begin{aligned}
		& \dot{\Pi}_t= 2\alpha \Pi_t- \beta^2 \Pi^2_t+q,\quad \Pi_0 = z_0>0, \quad q>0.  
	\end{aligned}
\end{equation}
Let $S$ be the positive solution to the corresponding algebraic Riccati equation:
\begin{equation} \label{equ:algebraic-Ricc}
	\begin{aligned}
		0 = 2\alpha S- \beta^2 S^2+q, \quad q>0.  
	\end{aligned}
\end{equation}
Denote $\Pi^e_t = \Pi_t -S $.
Subtracting \eqref{equ:algebraic-Ricc} from \eqref{equ:differential-Ricc} yields 
\begin{equation}\label{equ:error-Ricc}
	\begin{aligned}
		 &\dot{\Pi}^e_t= 2\alpha \Pi^e_t- \beta^2 \Pi^2_t + \beta^2 S^2 \\
		 &= 2\alpha \Pi^e_t- 2\beta^2S(\Pi_t-S)- \beta^2 \Pi^2_t + \beta^2 S^2 + 2\beta^2S(\Pi_t-S)\\
		& =2(\alpha-\beta^2 S) \Pi^e_t- \beta^2(\Pi^e_t)^2,\quad \Pi^e_0 = z_0- S, \quad q>0.  
	\end{aligned}
\end{equation}
If $\Pi^e_t \neq 0$ holds for all $t \in [0,T]$, we introduce $I^{e}_t = (\Pi^e_t)^{-1}$. Substituting $(I^{e}_t)^{-1}$ for  $\Pi^e_t$ in \eqref{equ:error-Ricc}  yields
\begin{equation}
	\begin{aligned}
		 &\dot{I}^e_t =-2(\alpha-\beta^2 S) I^e_t+ \beta^2,\quad I^e_0 = (z_0- S)^{-1}.  
	\end{aligned}
\end{equation} See also \cite{riccatiExplicitSol1998}.
Therefore $$I_t^e = \frac{e^{-2(\alpha-\beta^2S)t}}{(z_0-S)} + \beta^2\int_0^t e^{-2(\alpha-\beta^2S)\tau}d\tau $$ and hence 
\[
	\Pi_t = \left(\frac{e^{-2(\alpha-\beta^2S)t}}{(z_0-S)} + \beta^2\int_0^t e^{-2(\alpha-\beta^2S)\tau}d\tau\right)^{-1} + S
\]
with $S = \sqrt{\frac{\alpha^2}{\beta^4}+ \frac{q}{\beta^2} } + \frac{\alpha}{\beta^2}$.

\end{document}